\documentclass{amsart}
\usepackage{graphicx}
\usepackage{amscd}
\usepackage[pagebackref]{hyperref}

\newcommand{\C}{{\mathbb C}}
\newcommand{\Z}{{\mathbb Z}}
\newcommand{\R}{{\mathbb R}}

\newcommand{\co}{\colon}

\newtheorem{theorem}{Theorem}[section]
\newtheorem{lemma}[theorem]{Lemma}
\newtheorem{proposition}[theorem]{Proposition}

\theoremstyle{definition}
\newtheorem{definition}[theorem]{Definition}
\newtheorem{example}[theorem]{Example}
\newtheorem{remark}[theorem]{Remark}

\begin{document}
\title[Separating sets, metric tangent cone and applications ]
{Separating sets, metric tangent cone and applications for
  complex algebraic germs}
\author{Lev Birbrair}
\address{Departamento de Matem\'atica, Universidade Federal do Cear\'a
(UFC), Campus do Pici, Bloco 914, Cep. 60455-760. Fortaleza-Ce,
Brasil} \email{birb@ufc.br}
\author{Alexandre Fernandes}
\address{Departamento de Matem\'atica, Universidade Federal do Cear\'a
(UFC), Campus do Pici, Bloco 914, Cep. 60455-760. Fortaleza-Ce,
Brasil} \email{alexandre.fernandes@ufc.br}
\author{Walter D.
  Neumann}
\address{Department of Mathematics, Barnard College,
  Columbia University, New York, NY 10027}
\email{neumann@math.columbia.edu}

\subjclass{} \keywords{bi-Lipschitz, isolated complex singularity}

\begin{abstract}
  An explanation is given for the initially surprising ubiquity of
  separating sets in normal complex surface germs. It is  shown that they
  are quite common in higher dimensions too. The relationship between
  separating sets and the geometry of the metric tangent cone of
  Bernig and Lytchak is described. Moreover, separating sets are used
  to show that the inner Lipschitz type need not be constant in a
  family of normal complex surface germs of constant topology.
\end{abstract}

\maketitle

\section{Introduction}

Given a complex algebraic germ $(X,x_0)$, a choice of generators
$x_1,\dots,x_N$ of its local ring gives an embedding of $(X,x_0)$ into
$(\C^N,0)$. It then carries two induced metric space structures: the
``outer metric'' induced from distance in $\C^N$ and the ``inner
metric'' induced by arc-length of curves on $X$.  In the Lipschitz
category each of these metrics is independent of choice of embedding:
different choices give metrics for which the identity map is a
bi-Lipschitz homeomorphism.  The inner metric, which is given by a
Riemannian metric off the singular set, is the one that interests us
most here. It is determined by the outer metric, so germs that are
distinguished by their inner metrics are distinguished by their outer
ones.

These metric structures have so far seen much more study in real
algebraic geometry than in the complex algebraic world.  In fact,
until fairly recently conventional wisdom was that bi-Lipschitz
geometry would have little to say for normal germs of complex
varieties. For example, it is easy to see that two complex curve germs
with the same number of components are bi-Lipschitz homeomorphic
(inner metric). So for curve germs bi-Lipschitz geometry is equivalent
to topology. The same holds for outer bi-Lipschitz geometry of plane
curves: two germs of complex curves in $\C^2$ are bi-Lipschitz
homeomorphic for the outer metric if and only if they are
topologically equivalent as embedded germs \cite{TP,F}. However, it
has recently become apparent that the bi-Lipschitz geometry of complex
surface germs is quite rich; for example, they rarely have trivial
geometry (in the sense of being bi-Lipschitz homeomorphic to a metric
Euclidean cone). We give here an explanation which shows that the same
holds in higher dimensions too. The particular bi-Lipschitz invariants
we will discuss are ``separating sets''.

Let $(X,x_0)$ be a germ of a $k$-dimensional semialgebraic set. A
\emph{separating set} of $(X,x_0)$ (see Section \ref{sec:separating sets})
is a subgerm $(Y,x_0)\subset (X,x_0)$ of dimension less than $k$ which
locally separates $X$ into two pieces $A$ and $B$ which are ``fat'' at
$x_0$ while $Y$ itself is ``thin'' (i.e., the $k$--dimensional densities at
$x_0$ of $A$ and $B$ are nonzero and the $(k-1)$--dimensional density at
$x_0$ of $Y$ is zero).

There are trivial ways a separating set can occur---for example as the
intersection of the components of a complex germ $(X,x_0)$ which is
the union of two irreducible components of equal dimension. The
intersection of the two components clearly separates $X$ and it is
thin because its real codimension is at least $2$.  Fig.~\ref{fig:1}
illustrates schematically a codimension 1 example of a separating set.
\begin{figure}[ht]
    \centering
\includegraphics[width=.4\hsize]{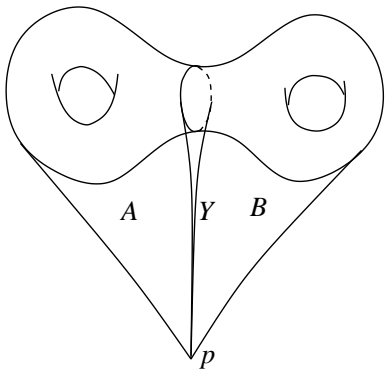}
    \caption{Separating set} 
    \label{fig:1}
\end{figure}
The interesting question is whether such separating sets exist ``in
nature''---for isolated singularities in particular.

For real algebraic singularities examples can be constructed (see
\cite{BF}), but they do not seem to arise very naturally. But for
normal complex surface singularities they had turned out to be surprisingly
common:  already the the simplest singularities, namely the Kleinian
surfaces singularities $A_k=\{(x,y,z)\in\C^3:x^2+y^2+z^{k+1}=0\}$,
have separating sets at the origin when $k>1$ (see \cite{BFN2}). This
paper is devoted to the investigation of this phenomena in all
complex dimensions $\ge 2$. We restrict to isolated complex singularities.

Our first result (Theorem \ref{th:sep}) is that \emph{if $X$ is a
  weighted homogeneous complex surface $\C^3$ with weights $w_1\geq
  w_2>w_3$ and if the zero set $X\cap \{z=0\}$ of the variable $z$ of
  lowest weight has more than one branch at the origin, then $(X,0)$
  has a separating set.}  This reproves that $A_k$ has a separating set
at the origin when $k>1$, and shows more generally that the same holds for
the Brieskorn singularity $$X(p,q,r):=\{(x,y,z)\in
\C^3~|~x^p+y^q+z^r=0\}$$ if $p\le q<r$ and
$\operatorname{gcd}(p,q)>1$.  

It also proves that for $t\ne0$ the singularity $X_t$ from the Brian\c
con-Speder family \cite{BS}
$$X_t=\{(x,y,z)\in\C^3 ~|~x^5+z^{15}+y^7z+txy^6=0\}$$ 
has a separating set at the origin.  On the other hand, we show
(Theorem \ref{bs_1}) that
$X_t$ does not have a separating set when $t=0$. Thus
\emph{the inner bi-Lipschitz type of a normal surface germ is not determined by
  topological type, even in a family of singularities of constant
  topological type} (and is thus also not determined by the resolution graph).

We also show (see Theorem \ref{th:septc}) that \emph{if the tangent
  cone $T_{x_0}X$ of an isolated complex singularity $(X,x_0)$ has a
  non-isolated singularity and the non-isolated locus is a separating
  set of $T_{x_0}X$, then $(X,x_0)$ has a separating set.}  

It follows, for instance (see \cite{BFN_rat}), that all quotient
surface singularities $\C^2/G$ (with $G\subset
\operatorname{GL}(2,\C)$ acting freelly on $\C^2\setminus\{0\}$) have
separating sets except the ones that are obviously conical
($\C^2/\mu_n$ with the group $\mu_n$ of $n$--th roots of unity acting
by multiplication) and possibly, among the simple singularities,
$E_6$, $E_7$, $E_8$ and the $D_n$ series. Moreover, for any $k$ one
can find cyclic quotient singularities with more than $k$ disjoint
non-equivalent separating sets.  Theorem \ref{th:septc} also easily
gives examples of separating sets for isolated singularities in higher
dimension.

It is natural to ask if the converse to \ref{th:septc} holds, i.e.,
separating sets in $(X,x_0)$ always correspond to separating sets in
$T_{x_0}X$, but this is not so: the tangent cone of the Brian\c
con-Speder singularity $X_t$, which has a separating set for $t\ne0$,
is $\C^2$.  But in Theorem \ref{th:nsc} we give necessary and
sufficient conditions for existence of a separating set in terms of
the ``metric tangent cone'' $\mathcal T_{x_0}X$, the theory of which
was recently developed by Bernig and Lytchak \cite{BL}.

$\mathcal T_{x_0}X$ is defined as the Gromov-Hausdorff limit as $t\to 0$
of the result of scaling the inner metric of the germ $(X,x_0)$ by
$\frac 1t$.  Another way of constructing $\mathcal T_{x_0}X$, and the one we
actually use, is as usual tangent cone of a ``normal re-embedding"
\cite{BM} of $X$ (for a complex germ, such a
normal re-embedding may only exists after forgetting complex
structure and considering $(X,x_0)$ as a real germ \cite{BFN_normal}).

\subsection*{Acknowledgements}
The authors acknowledge research support under the grants: CNPq grant
no 300985/93-2 (Lev Birbriar),  CNPq
grant no 300685/2008-4  (Alexandre Fernandes) and  NSF grant no.\
  DMS-0456227 (Walter Neumann).  They are greatful for useful
  correspondence with Bruce Kleiner and Frank Morgan. Neumann also
acknowledges the hospitality of the Max-Planck-Institut
  f\"ur Mathematik in Bonn for part of the work.

\section{Separating sets}\label{sec:separating sets}
Let $X\subset\R^n$ be a $k$-dimensional rectifiable subset. Recall
that the inferior and superior $k$--densities of $X$ at the point
$x_0\in\R^n$ are defined by:
$$\underline\Theta^k(X,x_0) = \lim_{\epsilon\to 0^+}
\inf\frac{\mathcal{H}^k(X\cap \epsilon B(x_0))}{\eta\epsilon^k}$$
and
$$\overline\Theta^k(X,x_0)=\lim_{\epsilon\to 0^+}
\sup\frac{\mathcal{H}^k(X\cap \epsilon B(x_0))}{\eta\epsilon^k}\,,$$
where $\epsilon B(x_0)$ is the $n$--dimensional ball of radius
$\epsilon$ centered at $x_0$, $\eta$ is the volume of the
$k$--dimensional unit ball and $\mathcal H^k$ is
$k$--dimensional Hausdorff measure in $\R^n$. If
$$\underline\Theta^k(X,x_0)=\theta=\overline\Theta^k(X,x_0)\,,$$
then $\Theta$ is called the $k$--dimensional density of $X$
at $x_0$ (or simply $k$--density at $x_o$).
\begin{remark} Recall that if $X\subset\R^n$ is a semialgebraic
  subset, then the above two limits are equal and the $k$--density of
  $X$ is well defined for any point of $\R^n$. Moreover, the vanishing
  or non-vanishing of these densities is a bi-Lipschitz invariant
  invariant, since a bi-Lipschitz homeomorphism clearly changes them
  by a factor that is bounded by $k$ and the Lipschitz
  constant.
\end{remark}
\begin{definition}
  Let $X\subset\R^n$ be a $k$-dimensional semialgebraic set and let
  $x_0\in X$ be a point such that the link of $X$ at $x_0$ is
  connected and the $k$--density of $X$ at $x_0$ is positive. A
  $(k-1)$--dimensional rectifiable subset $Y\subset X$ with
  $x_0\in Y$ is called a \emph{separating set of $X$ at $x_0$} if (see
  Fig.~\ref{fig:1})
\begin{itemize}
\item for some small $\epsilon>0$ the subset
  $\bigl(\epsilon B(x_0)\cap X\bigr)\setminus Y$ has at least two connected
  components $A$ and $B$,
\item the superior $(k-1)$--density of $Y$ at $x_0$ is  zero,
\item the inferior $k$--densities of $A$ and $B$ at $x_0$ are nonzero.
\end{itemize}
More generally, we need only require that the above is true locally,
in the sense that it holds after replacing $X$ by the union of
$\{x_0\}$ and a neighborhood of $Y\setminus \{x_0\}$ in
$X\setminus\{x_0\}$. For simplicity of exposition we will leave to the
reader to check that our results remain correct with this more general
definition.
\end{definition}
\begin{proposition}[Lipschitz invariance of separating sets] Let $X$
  and $Z$ be two real semialgebraic sets. If there exists a
  bi-Lipschitz homeomorphism of germs $F\co (X,x_0)\rightarrow
  (Z,z_0)$ with respect to the inner metric, then $X$ has a separating
  set at $x_0\in X$ if and only if $Z$ has a separating set at $z_0\in
  Z$.
\end{proposition}
\begin{proof} The result would be immediate if separating sets were
  defined in terms of the inner metrics on $X$ and $Z$. So we must
  show that separating sets can be defined this way.

  Let $X\subset\R^n$ be a connected semialgebraic subset. Consider the
  set $X$ equipped with the inner metric and with the Hausdorff
  measure $\mathcal{H}_X^k$ associated to this metric. Let $Y\subset
  X$ be a $k$-dimensional rectifiable subset. We define the inner
  inferior and superior densities of $Y$ at $x_0\in X$ with respect to inner
  metric on $X$ as follows:
$$\underline\Theta^k(X,Y,x_0)=\lim_{\epsilon\to 0^+}
\inf\frac{\mathcal{H}_X^k(Y\cap \epsilon B_X(x_0))}{\eta\epsilon^k}$$
and
$$\overline\Theta^k(X,Y,x_0)=\lim_{\epsilon\to
  0^+}\sup\frac{\mathcal{H}_X^k(Y\cap
  \epsilon B_X(x_0))}{\eta\epsilon^k}\,,$$
where $\epsilon B_X(x_0)$
denotes the closed ball in $X$ (with respect to the inner metric) of
radius $\epsilon$ centered at $x_0$.
The fact that separating sets can be defined using the inner metric
now follows from the following proposition, completing the
proof.\end{proof}
\begin{proposition}\label{proposition1} Let $X\subset\R^n$ be a
  semialgebraic connected subset. Let $W\subset X$ be a
  $k$-dimensional rectifiable subset and $x_0\in X$. Then, there exist
  two positive constants $\kappa_1$ and $\kappa_2$ such that:
$$\kappa_1\underline\Theta^k(X,W,x_0)\leq
\underline\Theta^k(W,x_0)\leq \kappa_2\underline\Theta^k(X,W,x_0)$$
and
$$\kappa_1\overline{\Theta^k}(X,W,x_0)\leq
\overline\Theta^k(W,x_0)\leq \kappa_2\overline\Theta^k(X,W,x_0)\,.$$
\end{proposition}
\begin{proof}
  If we used the outer metric instead of the inner metric in the
  definition of $\overline{\Theta^k}(X,W,x_0)$ and
  $\underline{\Theta^k}(X,W,x_0)$ we'd just get
  $\overline\Theta^k(W,x_0)$ and $\underline\Theta^k(W,x_0)$. Thus the
  proposition follows immediately from the Kurdyka's ``Pancake
  Theorem'' (\cite{Ku}, \cite{BM}) which says that if $X\subset\R^n$
  is a semialgebraic subset then there exists a finite semialgebraic
  partition $X=\bigcup_{i=1}^{l} X_i$ such that each $X_i$ is a
  semialgebraic connected set whose inner metric and outer (Euclidean)
  metric are bi-Lipschitz equivalent.
\end{proof}

The following Theorem shows that the germ of
an isolated complex singularity which has a separating set cannot be
\emph{metrically conical}, i.e., bi-Lipschitz homeomorphic to the
Euclidean metric cone on its link.
\begin{theorem}\label{th:cone}
  Let $(X, x_0)$ be a ($n+1$)-dimensional metric cone whose base is a
  compact connected Lipschitz manifold (possibly with
  boundary). Then, $X$ does not have a separating set at $x_0$.
\end{theorem}
\begin{proof} Let $M$ be an $n$-dimensional compact connected
  Lipschitz manifold with boundary. For convenience of exposition we
  will suppose that $M$ is a subset of the Euclidean sphere
  $S^{k-1}\in\R^k$ centered at $0$ and with radius $1$ and $X$ the
  cone over $M$ with vertex at the origin $0\in\R^k$.  Suppose that
  $Y\subset X$ is a separating set, so $X\setminus Y=A\cup B$ with $A$
  and $B$ open in $X\setminus Y$; the $n$--density of $Y$ at $0$ is
  equal to zero and the inferior $(n+1)$--densities of $A$ and $B$ at
  $0$ are unequal to zero. In particular, there exists $\xi >0$ such
  that these inferior densities of $A$ and $B$ at $0$ are bigger than
  $\xi$. For each $t>0$, let $\rho_t\co X\cap tD^k\rightarrow X$ be
  the map $\rho_t(x)=\frac{1}{t}x$, where $tD^k$ is the ball about
  $0\in \R^k$ of radius $t$.  Denote $Y_t=\rho_t(Y\cap tD^k)$,
  $A_t=\rho_t(A\cap tD^k)$ and $B_t=\rho_t(B\cap tD^k)$. Since the
  $n$--density of $Y$ at $0$ is equal to zero, we have:
$$\lim_{t\to
    0^+}\mathcal{H}^{n}(Y_t)=0\,.$$
Also, since the inferior
  densities of $A$ and $B$ at $0$ are bigger than $\xi$, we have that
  $\mathcal{H}^{n+1}(A_t)>\xi$ and $\mathcal{H}^{n+1}(B_t)>\xi$ for all
  sufficiently small $t>0$.

  Let $r$ be a radius such that $X\cap rD^k$ has volume $\le\xi/2$ and
  denote by $X'$, $A'_t$, $B'_t$, $Y'_t$ the result of removing from
  each of 
  $X$, $A_t$, $B_t$, $Y_t$ the intersection with the interior of the
  ball $rB^k$. Then $X'$ is a Lipschitz $(n+1)$--manifold (with
  boundary), $A'_t$ and $B'_t$ subsets of $(n+1)$--measure $>\xi/2$
  separated by $Y_t$ of arbitrarily small $n$--measure.

The following lemma then gives the contradiction to
complete the proof.
\end{proof}
\begin{lemma}\label{lipschitz_manifold} Let $X'$ be a $(n+1)$-dimensional
  compact connected  Lipschitz manifold with boundary. Then, for
  any $\xi>0$ there exists $\epsilon>0$ such that if $Y'\subset X'$ is a
  $n$-dimensional rectifiable subset with
  $\mathcal{H}^{n}(Y')<\epsilon$, then $X'\setminus Y'$ has a connected
  component $A$ of $(n+1)$--measure exceeding
  $\mathcal{H}^{n+1}(X')-\xi/2$ (so any remaining components have total measure
  $<\xi/2$).
\end{lemma}
\begin{proof}
  If $X'$ is bi-Lipschitz homeomorphic to a ball then this follows
  from standard isoperimetric results: for a ball the isoperimetric
  problem is solved by spherical caps normal to the boundary (Burago
  and Maz'ja \cite{burago-mazja} p.\ 54, see also Hutchins
  \cite{hutchins}). The isoperimetric problem is often formulated in
  terms of currents, in which case one uses also that the mass of
  the current boundary of a region is less than or equal to the
  Hausdorff measure of the topological boundary (\cite{federer} 4.5.6
  or \cite{morgan} Section 12.2).

  Let $\{ T_i \}_{i=1}^{m}$ be a cover of $X'$ by subsets which are
  bi-Lipschitz homeomorphic to balls and such that
$$T_i\cap T_j\neq\emptyset ~\Rightarrow ~\mathcal{H}^{n+1}(T_i\cap
T_j)>0.$$ Without loss of generality we may assume
$$\xi/m<\min\{\mathcal{H}^{n+1}(T_i\cap T_j)~|~T_i\cap T_j\ne
\emptyset\}\,.$$

Since $T_i$ is bi-Lipschitz homeomorphic to a ball there
exists $\epsilon_i$ satisfying the conclusion of this lemma for
$\xi/m$.  Let $\epsilon=\min(\epsilon_1,\dots,\epsilon_m)$. So if $Y'\subset X'$ is an $n$--dimensional
rectifiable subset such that $\mathcal{H}^n(Y')<\epsilon$, then
for each $i$ the largest component $A_i$ of $T_i\setminus Y'$ has
complement $B_i$ of measure $<\xi/2m$.

We claim $\bigcup_{i=1}^mA_i$ is connected. It suffices to show that $$T_i\cap
T_j\ne \emptyset~\Rightarrow~A_i\cap A_j\ne\emptyset\,.$$ So suppose
$T_i\cap T_j\ne \emptyset$. Then $B_i\cup B_j$ has measure less than
$\xi$, which is less than $\mathcal H^n(T_i\cap T_j)$, so $T_i\cap
T_j\not\subset B_i\cup B_j$. This is equivalent to $A_i\cap A_j\ne \emptyset$.

Thus there exists a connected component $A$ of $X'\setminus Y'$ which
contains $\bigcup_{i=1}^mA_i$. Its complement $B$ is
a subset of
$\bigcup_{i=1}^m B_i$ and thus has measure less than $\xi/2$.
\end{proof}

\section{Separating sets in normal surface singularities}

\begin{theorem}\label{th:sep} Let $X\subset\C^3$ be a weighted homogeneous
  algebraic surface with respect to the weights $w_1\geq w_2>w_3$ and
  with an isolated singularity at $0$. If
  $\bigl(X\setminus\{0\}\bigr)\cap\{z=0\}$ is not connected, then $X$
  has a separating set at $0$.
\end{theorem}
\begin{example}
  This theorem applies to the Brieskorn singularity $$X(p,q,r):=\{(x,y,z)\in
  \C^3~|~x^p+y^q+z^r=0\}$$
if $p\le q<r$ and $\operatorname{gcd}(p,q)>1$. In particular it is
not metrically conical. This was known for a different reason by
\cite{BFN1}: a weighted homogeneous surface singularity (not
necessarily hypersurface) whose two lowest weights are distinct is not
metrically conical.
\end{example}
\begin{proof}[Proof of Theorem \ref{th:sep}]
  Since $X$ is weighted homogeneous, the intersection $X\cap S^5$ is
  transverse and gives the singularity link.  By assumption, $
  \bigl(X\cap S^5\bigr)\cap\{ z=0\}$ is the disjoint union of two
  nonempty semialgebraic closed subsets $\widetilde{A}$ and
  $\widetilde{B}\subset \bigl(X\cap S^5\bigr)\cap\{ z=0\}$.  Let
  $\widetilde{M}$ be the conflict set of $\widetilde{A}$ and $\widetilde{B}$ in
  $X\cap S^5$, i.e.,
$$\widetilde M:=\{p\in X\cap S^5 ~|~ d(p,\widetilde A)=d(p,\widetilde B)\}\,,$$
where $d(\cdot,\cdot)$ is the standard metric on $ S^5$
(euclidean metric in $\C^3$ gives the same set).  Clearly, $\widetilde{M}$
is a compact semialgebraic subset and there exists $\delta>0$ such
that $d(\widetilde{M},\{ z=0\})>\delta$. Let $M=\C^*\widetilde M\cup\{0\}$
(the closure of the union of $\C^*$--orbits through $\widetilde{M}$). Note
that the $\C^*$--action restricts to a unitary action of $S^1$, so the
construction of $\widetilde M$ is invariant under the $S^1$--action, so
$M=\R^*\widetilde M$, and is therefore $3$--dimensional. It is
semi-algebraic by the Tarski-Seidenberg theorem.  We will use the
weighted homogeneous property of $M$ to show $\dim(T_0M)\leq 2$, where
$T_0M$ denotes the tangent cone of $M$ at $0$, from which will follow
that $M$ has zero $3$--density. In fact, we
will show that $T_0M\subset \{x=0,y=0\}$.

Let $T\colon \widetilde{M} \times [0,+\infty)\rightarrow
M$ be defined by:
$$T((x,y,z),t)=(t^{\frac{w_1}{w_3}}x,t^{\frac{w_2}{w_3}}y,tz).$$
Clearly, the restriction $T|_{ \widetilde{M} \times (0,+\infty)}\colon
\widetilde{M} \times (0,+\infty)\rightarrow M\setminus\{0\}$ is a
bijective semialgebraic map. Let $\gamma\colon [0,\epsilon)\rightarrow
M$ be a semianalytic arc; $\gamma(0)=0$ and $\gamma'(0)\neq 0$. We
consider $\phi(s)=T^{-1}(\gamma(s))$ for all $s\neq 0$.  Since $\phi$
is a semialgebraic map and $M$ is compact,
$\displaystyle\lim_{s\to 0}\phi(s)$ exists and belongs to $M\times
\{0\}.$ For the same reason, $\displaystyle\lim_{s\to
  0}\phi^{\prime}(s)$ also exists and is nonzero. Therefore, the arc
$\phi$ can be extended to $\phi\colon [0,\epsilon)\rightarrow
\widetilde{M} \times [0,+\infty)$ such that $\phi(0)\in \widetilde{M}\times
\{0\}$ and $\phi'(0)$ exists and is nonzero. We can take the
$[0,\infty)$ component of $\phi$ as parameter and write
$\phi(t)=((x(t),y(t),z(t),t)$.  Then $\gamma(t)=(t^{w_1/w_3}x(t),
t^{w_2/w_3}y(t), tz(t))$, so
\begin{eqnarray*}
  \lim_{t\to 0^+}\frac{\gamma(t)}{t} &=&
\left(\lim_{t\to 0}\frac{t^{\frac{w_1}{w_3}}}{t}x(t)\,,~
\lim_{t\to 0}\frac{t^{\frac{w_2}{w_3}}}{t}y(t)\,,~\lim_{t\to 0}z(t)\right) \\
  &=& (0,0,z(0))\,.
\end{eqnarray*}
This is a nonzero vector (note $|z(0)|>\delta$) in the set
$\{x=0,y=0\}$, so we obtain that $$T_0M\subset \{x=0,y=0\}.$$

Since $M$ is a $3$-dimensional semialgebraic set and $\dim(T_0M)\leq
2$, we obtain that the $3$-dimensional density of $M$ at $0$ is equal
to zero  (\cite{KR}).

Now, we have the following decomposition: $$X\setminus M=A\cup B\,,$$
where $\widetilde{A}\subset A$, $\widetilde{B}\subset B$, $A$ and $B$ are
$\C^*$--invariant and $A\cap B=\emptyset$. Since $A$ and $B$ are
semialgebraic sets, the $4$--densities $\operatorname{density}_4(A,0)$
and $\operatorname{density}_4(B,0)$ are defined. We will show that
these densities are nonzero. It is enough to prove that
$\dim_\R(T_0A)=4$ and $\dim_\R(T_0B)=4$. Let $\Gamma\subset A$ be a
connected component of $A\cap\{ z=0\}$. Note that
$\bar\Gamma=\Gamma\cup\{0\}$ is a complex algebraic curve.
We will show that $T_0A$ contains the set $\{(x,y,v)~|~ (x,y,0)\in
\bar\Gamma, v\in \C\}$ if $w_1=w_2$ (note that $\bar\Gamma$ is the line
through $(x,y,0)$ in this
case) or either the $y$--$z$ or the $x$--$z$ plane if $w_1<w_2$.

Given a smooth point $(x,y,0)\in \Gamma$ and $v\in \C$, we may choose
a smooth arc $\gamma\colon [0,\epsilon)\to A$ of the form
$\gamma(t)=(\gamma_1(t),\gamma_2(t), t^m\gamma_3(t))$ with
$(\gamma_1(0),\gamma_2(0))=(x,y)$ and $\gamma_3(0)=v$. Then, using the
$\R^*$--action, we transform this arc to the arc
$\phi(t)=t^j\gamma(t)$ with $j$ chosen so $jw_3+m=jw_2$. Now
$\phi(t)=(t^{jw_1}\gamma_1(t), t^{jw_2}\gamma_2(t),
t^{jw_2}\gamma_3(t))$ is a path in $A$ starting at the origin. Its
tangent vector $\rho$ at $t=0$,
$$\rho=\lim_{t\to 0+}\frac{\phi(t)}{t^{jw_2}}\,,$$
is $\rho=(x,y,v)$ if $w_1=w_2$ and $\rho=(0,y,v)$ if $w_1>w_2$. If
$w_1>w_2$ and $y=0$ then the same argument, but with $j$ chosen with
$jw_3+m=jw_1$, gives $\rho=(x,0,v)$. This proves our claim and
completes the proof that $T_0A$ has real dimension $4$. The proof for
$T_0B$ is the same.  
\end{proof}

\section{The Brian\c con-Speder example}

For each $t\in\C$, let $X_t=\{(x,y,z)\in\C^3 ~|~
x^5+z^{15}+y^7z+txy^6=0 \}$. This $X_t$ is weighted homogeneous
with respect to weights $(3,2,1)$ and has an isolated
singularity at $0\in\C^3$.

\begin{theorem}\label{bs_1} $X_t$ has a separating set at $0$ if $t\ne
  0$ but does not have a separating set at $0$ if $t=0$.
\end{theorem}
\begin{proof}
  As already noted, for $t\ne 0$ Theorem \ref{th:sep} applies, so
  $X_t$ has a separating set.  So from now on we take $t=0$. Denote
  $X:=X_0$.  In the following, for each sufficiently small
  $\epsilon>0$, we use the notation $$X^{\epsilon}=\{ (x,y,z)\in X ~|~
  \epsilon |y|\leq |z| \leq \frac{1}{\epsilon}|y| \}.$$ We need a
  lemma.
\begin{lemma}\label{lemma1} $X^{\epsilon}$ is
  metrically conical at the origin with connected link.
\end{lemma}
\begin{proof}
Note that the lemma makes a statement about the germ of $X^\epsilon$
at the origin. We will restrict to the part of $X^\epsilon$ that
lies in a suitable closed neighborhood of the origin.

Let $P\co \C^3\rightarrow\C^2$ be the orthogonal projection
$P(x,y,z)=(y,z)$. The restriction $P_X$ of $P$ to $X$ is a $5$-fold
cyclic branched covering map branched along
$\{(y,z)~|~z^{15}+y^7z=0\}$. This is the union of the $y$--axis in
$\C^2$ and the seven curves $y=\zeta z^2$ for $\zeta$ a $7$--th root
of unity. These seven curves are tangent to the $z$--axis.

Let $$C^{\epsilon}=\{ (y,z)\in \C^2 ~|~ \epsilon |y|\leq |z| \leq
\frac{1}{\epsilon}|y| \} .$$ Notice that no part of the branch locus
of $P_X$ with $|z|<\epsilon$ is in $C^{\epsilon}$. In particular, if
$D$ is a disk in $\C^2$ of radius $<\epsilon$ around $0$, then the map
$P_X$ restricted to $X^\epsilon$ has no branching over this disk. We
choose the radius of $D$ to be $\epsilon/2$ and denote by $Y$ the part
of part of $X^\epsilon$ whose image lies inside this disk.
Then $Y$ is a covering of $C^\epsilon\cap D$, and to complete the
proof of the lemma we must show it is a connected covering space and
that the covering map is bi-Lipschitz.

Since it is a Galois covering with group $\Z/5$, to show it is a
connected cover it suffices to show that there is a closed curve in
$C^\epsilon\cap D$ which does not lift to a closed curve in
$Y$. Choose a small constant $c\le \epsilon/4$ and consider the curve
$\gamma\colon[0,1]\to C^\epsilon\cap D$ given by $\gamma(t)=(ce^{2\pi
  it}, c)$. A lift to $Y$ has $x$--coordinate $(c^{15}+c^8e^{14\pi i
  t})^{1/5}$, which starts close to $c^{8/5}$ (at $t=0$) and ends close
to $c^{8/5}e^{(14/5)\pi i}$ (at $t=1$), so it is not a closed curve.

To show that the covering map is bi-Lipschitz, we note that locally
$Y$ is the graph of the implicit function $(y,z)\mapsto x$ given by
the equation $x^5+z^{15}+y^7z=0$, so it suffices to show that the
derivatives of this implicit function are bounded. Implicit
differentiation gives
$$\frac{\partial x}{\partial  y}
=-\frac{7y^6z}{5x^4}\,,\quad\frac{\partial x}{\partial
  z}=-\frac{15z^{14}+y^7}{5x^4}\,.$$
 It is easy to see that there exists
$\lambda >0$ such that
$$|15z^{14}+ y^7|\leq \lambda|z|^4, \
|y^7|\leq\lambda |z^{14}+y^{7}| \ \mbox{and} \
|15z^{14}+y^7|\leq\lambda |z^{14}+y^7|,$$
for all $(y,z)\in
C^{\epsilon}\cap D.$ We then get
$$\left|\frac{\partial x}{\partial
    y}\right|^5=\frac{7^5|y^{30}z^5|}{5^5|z^{14}+y^7|^4|z|^4}\le\frac{7^5}{5^5}\lambda^4|y^2z|<\frac{7^5\lambda^4\epsilon^3}{5^52^3}$$
and
$$\left|\frac{\partial x}{\partial
    z}\right|^5=\frac{|15z^{14}+y^7|^5}{5^5|z^{14}+y^7|^4|z|^4}\le \frac{\lambda^5}{5^5}\,,$$
completing the proof.
\end{proof}

We now complete the proof of Theorem \ref{bs_1}. Let us suppose that
$X$ has a separating set. Let $A,B,Y\subset X$ be subsets satisfying:
\begin{itemize}
\item for some small $\epsilon>0$ the subset $[\epsilon B(x_0)\cap
  X]\setminus Y$ is the union of relatively open subsets  $A$ and $B$,
\item the $3$-dimensional density of $Y$ at $0$
 is equal to zero,
\item the $4$-dimensional inferior densities of $A$ and $B$ at $0$
 are unequal to zero.
\end{itemize}
Set $$N^{\epsilon}=\{(x,y,z)\in\C^3 ~|~ |z|\leq\epsilon |y| \ \mbox{or}
\ |y|\leq\epsilon |z| \}.$$ For each subset $H\subset\C^3$ we
denote $$H^{\epsilon}=H\cap [\C^3\setminus N^{\epsilon}].$$

In this step, it is valuable to observe that there exists a positive constant
$K$ (independent of $\epsilon$) such that
\begin{equation}\label{volume_estimative}
\mathcal{H}^4(X\cap N^{\epsilon}\cap B(0,r)) \leq K  \epsilon r^4
\end{equation}
for all $0< r\leq 1$ (see, e.g., Comte-Yomdin \cite{CY}, chapter
5). By definition, the $4$-dimensional inferior density of $A$ at $0$
is equal to
$$\liminf_{r\to 0^+}\left(\frac{\mathcal{H}^4(A^{\epsilon}\cap
    B(0,r))}{r^4}
+\frac{\mathcal{H}^4(A\cap N^{\epsilon}\cap B(0,r))}{r^4}\right)$$
Then, if $\epsilon>0$ is sufficiently small, we can use inequality
\eqref{volume_estimative} in order to show that the $4$-dimensional
inferior density of $A^{\epsilon}$ is a positive number. In a similar
way, we can show that if $\epsilon>0$ is sufficiently small, then the
$4$-dimensional inferior density of $B^{\epsilon}$ at $0$ is a
positive number. These facts are enough to conclude that
$Y^{\epsilon}$ is a separating set of $X^{\epsilon}$.  But in view of
 Lemma \ref{lemma1} this contradicts Theorem \ref{th:cone}.
\end{proof}

\section{Metric Tangent Cone and separating sets}

Given a closed and connected semialgebraic subset $X\subset\R^m$
equipped with the inner metric $d_X$, for any point $x\in X$, we
denote by $\mathcal{T}_xX$ the metric tangent cone of $X$ at $x$; see
Bernig and Lytchak \cite{BL}. Recall that  the metric tangent cone of
a metric space $X$ at a point $x\in X$ is defined as the
Gromov-Hausdorff limit
$$\mathcal{T}_xX=\lim_{t\to 0+}(B(x,t),\frac{1}{t}d_X)$$ 
where $\frac{1}{t}d_X$ is the distance on $X$ divided by
$t$. Bernig and Lytchak show that for a semialgebraic set the metric
tangent cone exists and is semialgebraic.  Moreover, a semialgebraic
bi-Lipschitz homeomorphism of germs induces a bi-Lipschitz equivalence
of their metric tangent cones (with the same Lipschitz
constant).

Recall that a connected semialgebraic set $X\subset\R^m$ is called
\emph{normally embedded} if the inner $d_X$ and the outer $d_e$
metrics on $X$ are bi-Lipschitz equivalent.

\begin{theorem}[\cite{BM}] Let $X\subset\R^m$ be a connected
  semialgebraic set. Then there exist a normally embedded
  semialgebraic set $\widetilde{X}\subset\R^q$ and a semialgebraic
  homeomorphism $p\colon\widetilde{X}\rightarrow X$ which is bi-Lipschitz
  with respect to the inner metric. $\widetilde{X}$, or more precisely the
  pair $(\widetilde{X},p)$, is called a normal embedding of $X$.
\end{theorem}

The following result relates the metric tangent cone of $X$ at $x$ and
the usual tangent cone of the normally embedded set.

\begin{theorem}[{\cite[Section 5]{BL}}]\label{bernig_lytchak} Let
  $X\subset\R^m$ be a closed and connected semialgebraic set and $x\in
  X$.  If $(\widetilde{X},p)$ is a normal embedding of $X$, then
  $T_{p^{-1}(x)}X$ is bi-Lipschitz homeomorphic to the metric tangent
  cone $\mathcal{T}_xX$.
\end{theorem}

\begin{theorem}\label{bernig_lytchak_2}If $(X_1,x_1)$ and $(X_2,x_2)$
  are germs of semialgebraic sets which are semialgebraically
  bi-Lipschitz homeomorphic with respect to the induced outer metric,
  then their tangent cones $T_{x_1}X_1$ and $T_{x_2}X_2$ are
  semialgebraically bi-Lipschitz homeomorphic.
\end{theorem}
\begin{proof} This is proved in \cite{BFN_normal}.  Without the
  conclusion that the bi-Lipschitz homeomorphism of tangent cones is
  semi-algebraic it is immediate from Bernig and Lytchak \cite{BL},
  since, as they point out, the usual tangent cone (which they call
  the subanalytic tangent cone) is the metric cone with respect to the
  outer (Euclidean) metric.
\end{proof}
Recall that a partition $\{X_i\}_1^k$ of $X$ is called an
\emph{L-stratification} if each $X_i$ is a Lipschitz manifold and for
each $X_i$ and for each pair of points $x_1,x_2\in X_i$ there exist
two neighborhoods $U_1$ and $U_2$ and a bi-Lipschitz homeomorphisms
$h\colon U_1\rightarrow U_2$ such that for each $X_j$ one has
$h(X_j\cap U_1)=X_j\cap U_2$. An L-stratification is called
\emph{canonical} if any other L-stratification can be obtained as a
refinement of this one. In \cite{BB} it is proved, by a slight
modification of Parusinski's Lipschitz stratification \cite{Par}, that
any semialgebraic set admits a canonical semialgebraic
L-stratification. The collection of $k$-dimensional strata of the
canonical L-stratification of $X$ is called the \emph{k-dimensional
  L-locus of $X$}. By Theorem \ref{bernig_lytchak}, the metric tangent
cone of a semialgebraic set admits a canonical L-stratification.

Let $M\subset\R^n$ be a semialgebraic subset of the unit 
sphere centered at the origin $0\in\R^n$. Let $C(M)$ be the straight cone
over $M$ with the vertex at the origin $0\in\R^n$. We say that a subset is
a \emph{separating subcone} of $C(M)$ if:
\begin{itemize}
\item it is a straight cone over a closed subset $N\subset M$ with 
  vertex at the origin $0\in\R^n$;
\item $M\setminus N$ is not connected.
\end{itemize}

\begin{example} 
  Consider the Brieskorn singularity defined by:
  $$X(a_1,\dots,a_n):=\{(z_1,\dots,z_n)\in \C^n ~|~ z_1^{a_1}+\dots+z_n^{a_n}=0\\,,$$ 
  with $a_1=a_2=a\ge 2$ and $a_k>a$ for $k>2$. The tangent cone at the
  origin is the union of the $a$ complex hyperplanes $\{z_1=\xi z_2\}$
  with $\xi$ an $a$-th root of $-1$. These intersect along the
  $(n-2)$--plane $V=\{z_1=z_2=0\}$. Thus, $V$ is a separating subcone
  of the tangent cone $T_0X(a_1,\dots,a_n)$. The following theorem
  shows that $V$ is the tangent cone of a separating set in
  $X(a_1,\dots,a_n)$ (a special case of this is again the $A_k$
  surface singularity for $k>1$).
\end{example}

\begin{theorem}\label{th:septc} Let $X$ be an $n$--dimensional
  closed semialgebraic set and let $x_0\in X$ be a point such that the
  link of $X$ at $x_0$ is connected and the $n$--density of $X$ at
  $x_0$ is positive. Any semialgebraic separating subcone of
  codimension $\geq 2$ in the tangent cone $T_{x_0}X$ contains the
  tangent cone of a separating set of $X$ at $x_0$.
\end{theorem}

\begin{proof} As usual we can suppose that the point $x_0$ is the
  origin. Recall $rB(0)$ means the ball of radius $r$ about
  $0$. Observe that the function 
  $$f(r)=d_{\rm Hausdorff}(T_0X\cap rB(0),X\cap rB(0))$$
  is semialgebraic, continuous and $f(0)=0$. By the 
  definition of the tangent cone one has
  $f(r)=ar^{\alpha}+o(r^{\alpha})$ for some $a>0$ and
  $\alpha>1$.

For a semialgebraic set $W\subset\R^N$ with $0\in W$, let
$U_W^{c,\alpha}$ be the $\alpha$-horn like neighborhood of $W$,
defined by:
  $$U_W^{c,\alpha}=\{x\in \R^N ~|~ d_e(x,W)< c\| x\|^{\alpha}\}.$$ 
  For some $c>0$ and sufficiently small $r>0$ one has $X\cap
  rB(0)\subset U_{T_0X}^{c,\alpha}\cap rB(0)$. We fix this $r$
  and replace $X$ by $X\cap rB(0)$, so $X\subset U_{T_0X}^{c,\alpha}$

  Let $Y\subset T_0X$ be a semialgebraic separating subcone of
  codimension $\geq 2$. We may assume $Y$ is closed. We then have a
  partition $$T_0X=A\cup Y\cup B\,,$$ where $A$ and $B$ are disjoint
  open subsets of $T_0X$ of positive $n$--density. We can assume $\bar
  A\cap \bar B=Y$ (if not, replace $A$ by $A\cup (Y\setminus (\bar
  A\cap\bar B))$). Then $U_{T_0X}^{c,\alpha}=U_A^{c,\alpha}\cup
  U_B^{c,\alpha}$, so $X\subset U_A^{c,\alpha}\cup
  U_B^{c,\alpha}$. Let
  $Z=\overline{U_A^{c,\alpha}}\cap\overline{U_B^{c,\alpha}}$ and
  $Y'=X\cap Z$. Then $X\setminus Y'$ is the disjoint union of the open
  sets $A':=(\overline{U_A^{c,\alpha}}\cap X)\setminus Z$ and
  $B':=(\overline{U_B^{c,\alpha}}\cap X)\setminus Z$.

  Now $T_0U_A^{c,\alpha}=\overline{A}$ and
  $T_0U_B^{c,\alpha}=\overline{B}$, so $T_0Z\subset
  \overline{A}\cap\overline{B}=Y$ so $T_0Y'\subset Y$. It follows that
  $T_0(A')=\overline A$ and $T_0(B')=\overline B$, so $A'$ and $B'$
  have positive $n$--density. Thus $\partial Y'$ separates $X$ into
  open sets of which at least two have positive $n$-density. Moreover
  $\partial Y'$ is an $(n-1)$--dimensional semialgebraic set (if $c$ is
  chosen generically) and, since its tangent cone has dimension $\le (n-2)$,
  its $(n-1)$--density is zero (\cite{KR}). So $\partial Y'$ is a
  separating set.
\end{proof}
\begin{remark} The Brian\c con-Speder example $X_t$ presented in Section 4
  has tangent cone equal to the $yz$--plane, which is nonsingular
 and thus does not have separating subcone in codimension $2$, but $X_t$
 nevertheless has a semialgebraic separating set at $0$ if $t\ne 0$.
\end{remark}

\begin{proposition} Let $X$ be a n-dimensional closed semialgebraic
  set and let $x\in X$ be a point such that the link of $X$ at $x$ is
  connected and the n-density of $X$ at $x$ is positive. If $X$ is
  normally embedded and has a semialgebraic separating set at $x$,
  then the tangent cone $T_xX$ contains a semialgebraic separating
  subcone of codimension $\geq 2$.
\end{proposition}

\begin{proof} Suppose that $X$ has a separating set $Y\subset X$ at
  $x$. Let $A,B\subset X$ such that
\begin{enumerate}
\item[a.] $A\cap B =\{x\}$;
\item[b.] $X\setminus Y=A\setminus\{x\}\cup B\setminus\{x\}$;
\item[c.] the $n$-densities of $A$ and $B$ at $x$ are positive.
\end{enumerate}

Recall the following notation: $$S_xZ=\{v\in T_xZ ~|~ |v|=1\}.$$ So
$C(S_xZ)=T_xZ$. Since the $(n-1)$-density of $Y$ at $x$ is equal to
zero, $S_xY$ has codimension at least two in
$S_xX$. Let us show that $S_xX\setminus S_xY$ is not connected. If
$S_xX\setminus S_xY$ were connected, then $(S_xA\setminus S_xY)\cap
(S_xB\setminus S_xY)\neq\emptyset$. Let $v\in (S_xA\setminus S_xY)\cap
(S_xB\setminus S_xY)$. Since $v\in S_xA$ and $v\in S_xB$, there exist
two semialgebraic arcs $\gamma_1\colon [0,r)\rightarrow A$ and
$\gamma_2\colon [0,r)\rightarrow B$ such that $$|\gamma_i(t)-x|=t ~
\mbox{and} ~ |\gamma_i(t)-x|=t ~ \forall ~ t\in [0,r)$$
and $$\lim_{t\to 0^+}\frac{\gamma_1(t)-x}{t}=v=\lim_{t\to
  0^+}\frac{\gamma_2(t)-x}{t}.$$ Since $\gamma_1(t)$ and $\gamma_2(t)$
belong to different components of $X\setminus Y$, any arc in $X$
connecting $\gamma_1(t)$ to $\gamma_2(t)$ meets $Y$. That is why
$$d_X(\gamma_1(t),\gamma_2(t))\geq d_X(\gamma_1(t),Y).$$ 
Since $X$ is normally embedded, we conclude that
$$\lim_{t\to 0^+}\frac{d_e(\gamma_1(t),Y)}{t}=0.$$ Thus, $v\in S_xY$.

Finally, the $n$-densities of $T_xA\setminus T_xY$ and $T_xB\setminus
T_xY$ are positive (e.g., by \cite[Proposition 1.2]{BL}), so the proof
is complete.
\end{proof}

\begin{theorem}\label{th:nsc} Let $X$ be a closed semialgebraic set
  and let $x\in X$ be a point such that the link of $X$ at $x$ is
  connected. Then $X$  has a semialgebraic separating set at $x$
  if, and only if, the metric tangent cone $\mathcal{T}_xX$ is
  separated by an $L$-locus of codimension $\geq 2$.
\end{theorem}

\begin{proof} This follows directly from Proposition 5.5 and Proposition 5.7.
\end{proof}

\end{document}